\documentclass[11pt]{amsart}

\usepackage[margin=1.1in]{geometry}
\usepackage{amsmath,amssymb,amsfonts,amsthm,mathtools}
\usepackage{mathrsfs}
\usepackage{enumitem}
\usepackage[colorlinks=true,citecolor=blue,linkcolor=blue,urlcolor=blue]{hyperref}

\numberwithin{equation}{section}

\newtheorem{theorem}{Theorem}[section]
\newtheorem{conjecture}{Conjecture}[section]
\newtheorem{proposition}[theorem]{Proposition}
\newtheorem{lemma}[theorem]{Lemma}
\newtheorem{corollary}[theorem]{Corollary}
\newtheorem{remark}[theorem]{Remark}

\DeclareMathOperator{\Ric}{Ric}

\newcommand{\Hess}{\nabla^2}

\newcommand{\ip}[2]{\left\langle #1,#2\right\rangle}
\newcommand{\norm}[1]{\left|#1\right|}

\title{Scalar-Flatness for Critical Metrics of the $L^2$-Scalar Curvature Functional in Dimensions $5\le n\le 9$}
\author{Heng Zhang}

\address{School of Mathematical Sciences, University of Science and Technology of China, Hefei, China}
\email{hengz@mail.ustc.edu.cn}
\date{}

\begin{document}

\begin{abstract}
Let $(M^n,g)$ be a complete Riemannian manifold of dimension $n\geq 5$ endowed with a critical metric of the quadratic scalar-curvature functional
$$
\mathcal S^2(g)=\int_M R_g^2\,dV_g .
$$
For $n\geq 10$, Catino, Mastrolia and Monticelli [J. Math. Pures Appl. 211 (2026), 103883] established that all complete noncompact critical metrics with finite energy are scalar-flat, and they conjectured that this scalar-flatness result holds for all dimensions $n\geq 5$. In this paper, we settle the conjecture by verifying its validity for the remaining dimension range $5\leq n\leq 9$.
\end{abstract}

\maketitle
\section{Introduction}
\subsection{Background}
The variational object of this paper is the squared scalar-curvature functional
\begin{equation}\label{S2intro}
        \mathcal S^2(g)=\int_M R_g^2\,dV_g .
\end{equation}
We work on smooth manifolds without boundary.  On a noncompact manifold, a critical metric means a critical point of \eqref{S2intro} with respect to compactly supported variations; equivalently, it is a smooth solution of the Euler--Lagrange system displayed below.  If $M$ has more than one connected component, all arguments are understood componentwise.

Quadratic curvature energies have long served as a testing ground for rigidity and compactness phenomena in Riemannian geometry.  For $n\ge4$, the usual decomposition of the curvature tensor relates the $L^2$-norm of the full curvature tensor to the three basic quantities
$$
        \int_M |W_g|^2\,dV_g,
        \qquad
        \int_M |\Ric_g|^2\,dV_g,
        \qquad
        \int_M R_g^2\,dV_g,
$$
where $W_g$, $\Ric_g$ and $R_g$ denote the Weyl, Ricci and scalar curvatures.  The broader variational theory of such functionals includes, for example, the work of Anderson, Catino, and Gursky--Viaclovsky; see \cite{Anderson1997,Anderson2001,Catino2014,Catino2015,GurskyViaclovsky2015,GurskyViaclovsky2016} and the references therein.  The functional \eqref{S2intro} is distinguished by the fact that its critical equation closes on the scalar curvature and the Ricci tensor alone.

More precisely, 
using the notation
$$
      \nabla_g^2 R := \operatorname{Hess}_g R,
        \qquad
        \Delta_g R:=\operatorname{tr}_g(\nabla_g^2 R).
$$
a critical metric of $\mathcal S^2$ satisfies
\begin{equation}\label{EL0intro}
        2R\Ric-2\Hess R+2\Delta R\,g=\frac12R^2g .
\end{equation}
Equivalently,
\begin{align}
        R\Ric - \Hess R &= \frac{3}{4(n-1)} R^2 g, \label{EL1}\\
        \Delta R &= \frac{n-4}{4(n-1)}R^2 . \label{EL2}
\end{align}
The second equation is the trace of the first; see \cite[Proposition~4.66]{Besse}.  We use the normalization of \cite[(1.1)--(1.2)]{CMM2026}.  In particular, when $n\ge5$, the scalar curvature of a solution is subharmonic.  This elementary trace feature is one of the main reasons why finite-energy assumptions are effective in the present problem.

Let us recall the rigidity picture for $\mathcal S^2$ which motivates the question considered here.  Catino--Mastrolia--Monticelli proved in \cite{CMM2023} that, in low dimensions and under suitable $L^q$ hypotheses in the noncompact setting, critical metrics of $\mathcal S^2$ are forced into the scalar-flat class; dimension four is exceptional because quadratic curvature integrals are scale invariant there, and the critical equation only gives harmonicity of $R$; the corresponding noncompact Liouville result is due to Yau \cite{Yau1976}.  In dimensions $n\ge5$, their earlier theorem gives scalar-flatness under an $L^q$ condition with $q$ below an explicit threshold $q_*>2$, together with a lower bound for the scalar curvature.  The finite-energy problem asks whether this lower-bound assumption can be eliminated at the endpoint most natural for the functional itself, namely
$$
        R_g\in L^2(M,g).
$$

For complete noncompact critical metrics, Catino--Mastrolia--Monticelli answered this question affirmatively in dimensions $n\ge10$ and conjectured the same conclusion for every $n\ge5$ \cite[Theorem~1.1 and Conjecture~1.2]{CMM2026}.

\begin{conjecture}[Catino--Mastrolia--Monticelli]\label{conj}
Let $(M^n,g)$ be complete and noncompact, with $n\ge5$.  If $g$ is a critical metric of $\mathcal S^2$ and $R_g\in L^2(M,g)$, then $R_g\equiv0$.  Consequently $g$ is a global minimizer of $\mathcal S^2$.
\end{conjecture}

The aim of this paper is to settle precisely the dimensional range not covered by the high-dimensional conformal argument of \cite{CMM2026}.

\subsection{Main result}

Our result is the following.

\begin{theorem}\label{thm:main}
Let \((M^n,g)\), \(5\leq n\leq 9\), be a complete critical metric of \(\mathcal S^2\) satisfying
$$
        \int_M R_g^2\,dV_g<\infty .
$$
Then \(R_g\equiv 0\).  In particular, \((M,g)\) is a global minimum of \(\mathcal S^2\).
\end{theorem}

The statement includes the compact case.  If $M$ is compact, integrating \eqref{EL2} gives the conclusion immediately because $n\ne4$.  Thus all work is in the complete noncompact case.

Together with \cite[Theorem~1.1]{CMM2026}, Theorem \ref{thm:main} gives the conjectured finite-energy scalar-flatness theorem in all dimensions in which it was proposed.

\begin{corollary}\label{cor:full-conjecture}
Conjecture \ref{conj} holds true.
\end{corollary}

\subsection{Idea of the proof}

We describe the proof in some detail, since the point of the paper is the change of conformal normalization in dimensions $5\le n\le9$.  The sign reduction in \cite[Lemma~5.1]{CMM2023}, combined with the strong maximum principle applied to \eqref{EL2}, says that a complete noncompact critical metric with $R\in L^q$, $1<q<\infty$, has, on each connected component, either $R\equiv0$ or $R<0$; the same reduction is recalled in \cite[Section~2]{CMM2026}.  Hence it remains only to exclude the second alternative.  We set
$$
        u=-R>0.
$$
Then \eqref{EL1}--\eqref{EL2} become
\begin{align}
        \Ric_g &= \frac{\Hess_g u}{u}-\frac{3}{4(n-1)}u g, \label{intro-uRic}\\
        \Delta_g u &= -\frac{n-4}{4(n-1)}u^2 . \label{intro-uLap}
\end{align}

The proof of \cite{CMM2026} in dimensions $n\ge10$ uses the conformal metric
$$
        \widetilde g=u^{6/(n-4)}g .
$$
That metric leads to a steady quasi-Einstein equation, and the sign of the resulting coefficient is compatible with known scalar-curvature nonnegativity results for steady Ricci solitons and steady quasi-Einstein metrics \cite{Chen2009,Wang2011}.  This provides a gradient estimate for $u$ after proving completeness of $\widetilde g$.  Conformal changes of this type also occur in the analysis of stationary vacuum metrics and in rigidity problems for stable minimal submanifolds \cite{Anderson2000,FischerColbrie,CMR2024,MagliaroMariRoingSavas2024}.  In the dimensions considered here, however, this same conformal metric does not have the sign structure needed for that argument.

Our replacement is the simpler metric
\begin{equation}\label{intro-hdef}
        h=ug .
\end{equation}
It corresponds to the choice $k=1/2$ in the conformal identity of \cite[Proposition~2.1]{CMM2026}.  With
$$
        f=\frac{n-4}{2}\log u,
        \qquad
        m=\frac{(n-4)^2}{10-n},
        \qquad
        \lambda=\frac{n-10}{8(n-1)},
$$
one obtains
\begin{equation}\label{intro-QE}
        \Ric_h+\Hess_h f-\frac1m df\otimes df=\lambda h .
\end{equation}
For $5\le n\le9$, the parameter $m$ is positive and finite, while $\lambda<0$.  Thus \eqref{intro-QE} is an expanding finite-dimensional Bakry--Emery, or quasi-Einstein, equation.  This places the problem in the framework of weighted comparison geometry \cite{BakryEmery,CaseShuWei,HePetersenWylie,Qian,WeiWylie}, rather than in the scalar-curvature argument used in the high-dimensional case.

The nontrivial preliminary step is to show that $h$ is complete.  Assuming incompleteness, a Hopf--Rinow limiting argument produces an $h$-unit-speed geodesic
$$
        \sigma:[0,T)\to M,
        \qquad T<\infty,
$$
which is minimizing on compact subintervals and has infinite length with respect to $g$.  If
$$
        v=u^{-\frac{10-n}{2(n-4)}},
        \qquad
        \alpha=\frac{n-4}{10-n},
$$
then $g=v^{2\alpha}h$, and the infinite $g$-length reads
$$
        \int_0^T v(\sigma(t))^\alpha\,dt=\infty .
$$
The second variation inequality along $\sigma$, together with the quasi-Einstein equation written in the form
$$
        \Ric_h-mv^{-1}\Hess_h v=\lambda h,
$$
yields a one-dimensional Hardy inequality for $v\circ\sigma$.  We prove a Logarithmic growth lemma showing that such an inequality prevents $v\circ\sigma$ from growing fast enough near $T$ to make the above integral diverge.  The dimension restriction enters through the elementary estimate
$$
        \alpha p_\rho<1,
        \qquad
        \rho=\frac{m}{n-1},
$$
where $p_\rho$ is determined by $\rho p_\rho(p_\rho+1)=1/4$.  This inequality is valid for $5\le n\le9$, and it forces completeness of $h$.

Once completeness is available, we use the weighted Laplacian
$$
        \Delta_f^h=\Delta_h-\langle\nabla_h f,\nabla_h\cdot\rangle_h
$$
on the complete weighted manifold $(M,h,e^{-f}dV_h)$.  For
$$
        w=-\frac{10-n}{2(n-4)}\log u,
$$
one has
$$
        \Delta_f^h w=-\lambda,
        \qquad
        \frac12\Delta_f^h |\nabla_h w|_h^2
        =|\Hess_h w|_h^2+|\nabla_h w|_h^2\bigl(m|\nabla_h w|_h^2+\lambda\bigr).
$$
A cutoff maximum principle based on finite-dimensional Bakry--Emery Laplacian comparison gives
$$
        |\nabla_h w|_h^2\le -\frac{\lambda}{m}.
$$
In the original metric this is exactly
\begin{equation}\label{intro-grad}
        |\nabla_g u|_g^2\le \frac1{2(n-1)}u^3 .
\end{equation}
Thus $u^{-1/2}$ grows at most linearly in the $g$-distance.  This gives a quadratic lower bound for $u$, and combining that bound with \eqref{intro-uLap} and the finite-energy assumption $u\in L^2(M,g)$ in a standard cutoff argument forces $u\equiv0$, contradicting the assumption $u>0$.

\subsection{Organization}

Section~\ref{sec:reduction} derives the expanding quasi-Einstein structure associated with $h=ug$.  Section~\ref{sec:log} proves the one-dimensional Logarithmic growth lemma.  Section~\ref{sec:completeness} establishes completeness of the conformal metric $h$.  Section~\ref{sec:maprinciple} proves the weighted maximum principle and the resulting gradient estimate \eqref{intro-grad}.  Section~\ref{sec:proofmain} completes the proof of Theorem \ref{thm:main} by the final $L^2$ cutoff argument.

\section{Reduction to an expanding quasi-Einstein structure with positive parameter}\label{sec:reduction}

The compact case of Theorem \ref{thm:main} is immediate: integrating \eqref{EL2} over $M$ gives
$$
        0=\frac{n-4}{4(n-1)}\int_M R^2\,dV_g,
$$
and hence $R\equiv 0$.  Henceforth $M$ is assumed noncompact.  Since all arguments are componentwise, we work on one connected component at a time.

We first record the sign reduction used in the sequel.

\begin{lemma}[Sign alternative]\label{lem:sign}
Let $(M^n,g)$, $n\geq 5$, be a complete noncompact critical metric of $\mathcal S^2$ with $R\in L^q(M)$ for some $1<q<\infty$.  Then $R\leq0$.  Consequently, on each connected component, either $R\equiv0$ or $R<0$.
\end{lemma}

\begin{proof}
The nonpositivity is \cite[Lemma~5.1]{CMM2023}.  Since $n\ge5$, equation \eqref{EL2} gives $\Delta R\ge0$ wherever $R$ is evaluated as a smooth function satisfying the critical equation.  Once $R\le0$, the strong maximum principle applied to \eqref{EL2} implies that either $R\equiv0$ or $R<0$ on each connected component.  This same reduction is recalled at \cite[Section~2]{CMM2026}.
\end{proof}

By Lemma \ref{lem:sign}, either $R\equiv0$ or $R<0$ on the connected component under consideration.  To prove the theorem, assume for contradiction that
$$
        R<0,
$$
and set
$$
        u=-R>0.
$$
Then \eqref{EL1}--\eqref{EL2} become
\begin{align}
        \Ric_g &= \frac{\Hess_g u}{u}-\frac{3}{4(n-1)}u g, \label{uRic}\\
        \Delta_g u &= -\frac{n-4}{4(n-1)}u^2 . \label{uLap}
\end{align}

Define the conformal metric
\begin{equation}\label{hdef}
        h=u g.
\end{equation}
Let
\begin{equation}\label{fmdef}
        f=\frac{n-4}{2}\log u,
        \qquad
        m=\frac{(n-4)^2}{10-n},
        \qquad
        \lambda=\frac{n-10}{8(n-1)} .
\end{equation}
For $5\leq n\leq 9$ one has $m>0$ and $\lambda<0$.

\begin{proposition}\label{prop:QE}
The metric $h=ug$ satisfies
\begin{equation}\label{QE-f}
        \Ric_h+\Hess_h f-\frac1m df\otimes df=\lambda h.
\end{equation}
Equivalently, if
\begin{equation}\label{vdef}
        v=u^{-\frac{10-n}{2(n-4)}},
\end{equation}
then
\begin{equation}\label{QE-v}
        \Ric_h-m v^{-1}\Hess_h v=\lambda h.
\end{equation}
Moreover,
\begin{equation}\label{weighted-volume}
        e^{-f}\,dV_h=u^2\,dV_g .
\end{equation}
\end{proposition}

\begin{proof}
We use the conformal formula \cite[Proposition~2.1]{CMM2026}.  In the notation of that proposition, for any $k\neq0$ with $(n-2)k\neq1$, the conformal metric
$$
        \widetilde g=u^{2k}g,
        \qquad
        F=[(n-2)k-1]\log u,
$$
satisfies
\begin{align*}
\Ric_{\widetilde g}+\Hess_{\widetilde g}F
&-\frac{1+2k-(n-2)k^2}{[(n-2)k-1]^2}\,dF\otimes dF  \\
&=\frac{(n-4)k-3}{4(n-1)}
  \exp\!\left(\frac{1-2k}{(n-2)k-1}F\right)\widetilde g .
\end{align*}
Taking $k=1/2$ gives $\widetilde g=h$, $F=f$, and the exponential factor is equal to $1$.  Moreover,
$$
        \frac{1+2k-(n-2)k^2}{[(n-2)k-1]^2}
        =\frac{10-n}{(n-4)^2}=\frac1m,
$$
while
$$
        \frac{(n-4)k-3}{4(n-1)}=\frac{n-10}{8(n-1)}=\lambda .
$$
This proves \eqref{QE-f}.  Since $f=-m\log v$, we have
$$
        \Hess_h f-\frac1m df\otimes df
        =-m\left(\Hess_h\log v+d\log v\otimes d\log v\right)
        =-m v^{-1}\Hess_h v,
$$
which gives \eqref{QE-v}.  Finally, because $h=ug$,
$$
        dV_h=u^{n/2}dV_g,
$$
and hence
$$
        e^{-f}dV_h=u^{-(n-4)/2}u^{n/2}dV_g=u^2dV_g .
$$
\end{proof}

It will be useful to introduce also
\begin{equation}\label{wdef}
        w=\log v=-\frac{10-n}{2(n-4)}\log u .
\end{equation}
Then $f=-mw$, and the weighted Laplacian associated with $(h,e^{-f}dV_h)$ is
\begin{equation}\label{Delf}
        \Delta_f^h \psi=\Delta_h\psi-\ip{\nabla_h f}{\nabla_h\psi}_h
        =\Delta_h\psi+m\ip{\nabla_h w}{\nabla_h\psi}_h .
\end{equation}

\begin{lemma}\label{lem:Delfw}
The function $w$ satisfies
\begin{equation}\label{Delfw}
        \Delta_f^h w=-\lambda .
\end{equation}
Consequently, if
$$
        G=\norm{\nabla_h w}_h^2,
$$
then
\begin{equation}\label{BochnerG}
        \frac12\Delta_f^h G
        =\norm{\Hess_h w}_h^2+G(mG+\lambda).
\end{equation}
\end{lemma}

\begin{proof}
Let
$$
        a_n=\frac{n-4}{4(n-1)},
        \qquad
        \beta=\frac{10-n}{2(n-4)},
$$
so that $w=-\beta\log u$ and $\Delta_g u=-a_nu^2$.  Since $h=ug=e^{2\varphi}g$ with $\varphi=\frac12\log u$, the Laplacian transforms as
$$
        \Delta_h \psi
        =u^{-1}\left(\Delta_g\psi+\frac{n-2}{2}\ip{\nabla_g\log u}{\nabla_g\psi}_g\right).
$$
Applying this to $w=-\beta\log u$ gives
\begin{align*}
        \Delta_h w
        &= -\beta u^{-1}\left(\Delta_g\log u+\frac{n-2}{2}\norm{\nabla_g\log u}_g^2\right) \\
        &= \beta a_n-\frac{10-n}{4}u^{-1}\norm{\nabla_g\log u}_g^2 .
\end{align*}
On the other hand,
$$
        m\norm{\nabla_h w}_h^2
        =m\beta^2u^{-1}\norm{\nabla_g\log u}_g^2
        =\frac{10-n}{4}u^{-1}\norm{\nabla_g\log u}_g^2 .
$$
Therefore
$$
        \Delta_f^h w=\Delta_h w+m\norm{\nabla_h w}_h^2
        =\beta a_n
        =\frac{10-n}{8(n-1)}=-\lambda .
$$
For \eqref{BochnerG}, use the weighted Bochner formula
$$
        \frac12\Delta_f^h G
        =\norm{\Hess_h w}_h^2+\bigl(\Ric_h+\Hess_h f\bigr)(\nabla_h w,\nabla_h w)
        +\ip{\nabla_h w}{\nabla_h\Delta_f^h w}_h .
$$
Since \eqref{Delfw} is constant, the last term vanishes.  From \eqref{QE-f} and $df=-m\,dw$,
$$
        \Ric_h+\Hess_h f=\lambda h+\frac1m df\otimes df
        =\lambda h+m\,dw\otimes dw .
$$
Thus
$$
        \bigl(\Ric_h+\Hess_h f\bigr)(\nabla_h w,\nabla_h w)
        =\lambda G+mG^2,
$$
which proves \eqref{BochnerG}.
\end{proof}

\section{A Logarithmic growth lemma}\label{sec:log}

The next elementary lemma is useful in the completeness argument.

\begin{lemma}[Logarithmic growth lemma]\label{lem:log}
Let $v>0$ be a $C^2$ function on $(0,T)$, where $0<T<\infty$.  Assume that there exist constants $\rho>0$ and $b\geq 0$ such that
\begin{equation}\label{log-assumption}
        \rho\int_0^T \frac{v''}{v}\phi^2\,dt
        \leq
        \int_0^T (\phi')^2\,dt+b\int_0^T \phi^2\,dt
\end{equation}
for every $\phi\in C_c^\infty(0,T)$.  Set
\begin{equation}\label{prho}
        p_\rho=\frac{\sqrt{1+\rho^{-1}}-1}{2},
\end{equation}
so that $\rho p_\rho(p_\rho+1)=1/4$.  Then
\begin{equation}\label{growth-conclusion}
        \limsup_{t\to T^-}
        \frac{\log v(t)}{\log \frac1{T-t}}
        \leq p_\rho .
\end{equation}
Consequently, for every $a>0$ such that $a p_\rho<1$, one has
\begin{equation}\label{integrability-conclusion}
        \int^{T} v(t)^a\,dt<\infty .
\end{equation}
\end{lemma}

\begin{proof}
Put $x=T-t$ and $V(x)=v(T-x)$.  Let
$$
        y=(\log V)' .
$$
Then $v''(t)/v(t)=V''(x)/V(x)=y'+y^2$.  Rewriting \eqref{log-assumption} in the $x$ variable gives
$$
        \rho\int_0^T (y'+y^2)\phi^2\,dx
        \leq
        \int_0^T (\phi')^2\,dx+b\int_0^T\phi^2\,dx .
$$
For compactly supported smooth $\phi$,
\begin{equation}\label{square-identity}
        \int (y'+y^2)\phi^2
        =\int (y\phi-\phi')^2-\int (\phi')^2 .
\end{equation}
Hence
\begin{equation}\label{square-bound}
        \rho\int_0^T (y\phi-\phi')^2\,dx
        \leq
        (1+\rho)\int_0^T(\phi')^2\,dx+b\int_0^T\phi^2\,dx .
\end{equation}

Fix $R\in(0,T)$ sufficiently small.  Given $0<r<R/8$, choose a smooth cutoff $\chi_{r,R}$ such that
$$
\begin{array}{ll}
        \chi_{r,R}=1 &\text{on }[2r,R/2],\\
        \chi_{r,R}=0 &\text{near }0\text{ and near }R,
\end{array}
$$
with
$$
        |\chi_{r,R}'|\leq C/r\quad\text{on }[r,2r],
        \qquad
        |\chi_{r,R}'|\leq C/R\quad\text{on }[R/2,R],
$$
and analogous second-derivative bounds.  The function
$$
        \phi(x)=x^{1/2}\chi_{r,R}(x)
$$
is smooth and compactly supported in $(0,T)$, because $\chi_{r,R}$ is identically zero near $x=0$.  Equivalently, one may start from the ideal piecewise-smooth cutoff and approximate it in $H_0^1(0,T)$.  Substitution in \eqref{square-bound} yields
$$
        \int_0^T(\phi')^2\,dx
        =\frac14\log\frac{R}{r}+O(1),
        \qquad
        \int_0^T\phi^2\,dx=O(R^2),
$$
where $R$ is fixed and the constants in $O(1)$ are independent of $r$.  Since $R$ is fixed, the term $bO(R^2)$ is also independent of $r$ and disappears after division by $\log(1/r)$ and passage to the limit $r\to0^+$.  On $[2r,R/2]$,
$$
        y\phi-\phi'=x^{1/2}\left(y-\frac1{2x}\right).
$$
Therefore
\begin{equation}\label{central-bound}
        \int_{2r}^{R/2} x\left(y-\frac1{2x}\right)^2\,dx
        \leq
        \frac{1+\rho}{4\rho}\log\frac{R}{r}+O(1).
\end{equation}
By Cauchy--Schwarz,
\begin{align*}
&\log V(2r)-\log V(R/2)+\frac12\log\frac{R/2}{2r} \\
&\qquad
= -\int_{2r}^{R/2}\left(y-\frac1{2x}\right)\,dx \\
&\qquad
\leq
\left|\int_{2r}^{R/2}\left(y-\frac1{2x}\right)\,dx\right| \\
&\qquad
\leq
\left(\int_{2r}^{R/2}x\left(y-\frac1{2x}\right)^2\,dx\right)^{1/2}
\left(\int_{2r}^{R/2}\frac{dx}{x}\right)^{1/2} \\
&\qquad
\leq
\sqrt{\frac{1+\rho}{4\rho}}\log\frac{R/2}{2r}+o\!\left(\log\frac1r\right)
\end{align*}
as $r\to0^+$.  Dividing by $\log(1/r)$ and letting $r\to0^+$ first gives
$$
        \limsup_{r\to0^+}\frac{\log V(2r)}{\log(1/r)}
        \leq
        \sqrt{\frac{1+\rho}{4\rho}}-\frac12
        =p_\rho .
$$
Replacing $2r$ by $s$ does not change the limsup, since
$\log(1/r)=\log(1/s)+\log 2$ and hence
$\log(1/r)/\log(1/s)\to1$ as $s\to0^+$.  Therefore
$$
        \limsup_{s\to0^+}\frac{\log V(s)}{\log(1/s)}
        \leq p_\rho .
$$
This is \eqref{growth-conclusion}.  If $a>0$ and $a p_\rho<1$, choose $\varepsilon>0$ with $a(p_\rho+\varepsilon)<1$.  From \eqref{growth-conclusion}, after decreasing the terminal interval if necessary, one has
$$
        v(t)\leq C_\varepsilon (T-t)^{-p_\rho-\varepsilon}
        \qquad\text{for }t\text{ sufficiently close to }T.
$$
Thus
$$
        \int^T v(t)^a\,dt
        \leq
        C\int^T (T-t)^{-a(p_\rho+\varepsilon)}\,dt<\infty .
$$
\end{proof}

\section{Completeness of the conformal metric \texorpdfstring{$h=ug$}{h=ug}}\label{sec:completeness}

We now prove that the conformal metric $h=ug$ is complete.  We first record a standard limiting-geodesic lemma in a form adapted to the present situation.

\begin{lemma}[Finite conformal geodesic with infinite background length]\label{lem:finite-geodesic}
Let $(M,g)$ be a connected complete Riemannian manifold, let $a\in C^\infty(M)$ be positive, and set $h=a g$.  If $h$ is incomplete, then there exist $0<T<\infty$ and an $h$-unit-speed geodesic
$$
        \sigma:[0,T)\to M
$$
such that:
\begin{enumerate}[label=\textup{(\roman*)}]
\item every compact subinterval of $[0,T)$ is minimizing for $h$;
\item the $g$-length of $\sigma$ is infinite, namely
$$
        \int_0^T |\dot\sigma(t)|_g\,dt=\infty .
$$
\end{enumerate}
Consequently, for every $\phi\in C_c^\infty(0,T)$,
\begin{equation}\label{index-ineq-general}
        \int_0^T\left[(n-1)(\phi')^2-\Ric_h(\dot\sigma,\dot\sigma)\phi^2\right]dt\geq 0 .
\end{equation}
\end{lemma}

\begin{proof}
Since $h$ is incomplete, the Hopf--Rinow theorem (equivalently the equivalence between metric completeness and geodesic completeness in finite-dimensional Riemannian geometry), gives a maximal $h$-unit-speed geodesic $\zeta:[0,T_*)\to M$ with $T_*<\infty$.  If $\zeta$ were contained in a compact subset of $M$, then $(\zeta,\dot\zeta)$ would remain in a compact subset of $TM$, because $|\dot\zeta|_h\equiv1$; the smooth geodesic equation for $h$ would then extend it past $T_*$ by the standard ODE extension theorem, a contradiction.  Thus $\zeta$ leaves every compact subset of $M$.  Since $M$ is connected, after joining the initial point of $\zeta$ to a fixed reference point $o\in M$, there is a finite $h$-length curve starting at $o$ and leaving every compact subset of $M$.

Let $B_r=B_r^g(o)$.  Since $(M,g)$ is complete, Hopf--Rinow implies that the closed $g$-balls $\overline{B_r}$, and hence the distance spheres $\partial B_r$, are compact.  Set
$$
        D(r)=d_h(o,\partial B_r).
$$
The preceding escaping curve has finite $h$-length, say at most $L_0$.  Since its $g$-distance from $o$ is unbounded, for every sufficiently large $r$ it meets $\partial B_r$ by continuity of $x\mapsto d_g(o,x)$.  Hence
\begin{equation}\label{Dr-bound}
        D(r)\leq L_0
\end{equation}
for all sufficiently large $r$.  Choose a sequence $r_i\to\infty$ for which \eqref{Dr-bound} holds.

For each $i$, choose $\eta_i\in C^\infty(M)$ satisfying
$$
        \eta_i\geq0,
        \qquad
        \eta_i\equiv0\ \text{on }\overline{B_{r_i}},
        \qquad
        \eta_i\equiv1\ \text{on }M\setminus B_{r_i+1}.
$$
Such a function exists by the smooth Urysohn lemma, since the closed sets $\overline{B_{r_i}}$ and $M\setminus B_{r_i+1}$ are disjoint.  Put
$$
        h_i=(a+\eta_i)g .
$$
The metric $h_i$ is complete.  Indeed, on the compact set $\overline{B_{r_i+1}}$ the positive function $a+\eta_i$ has a positive minimum $c_i>0$, while on $M\setminus B_{r_i+1}$ one has $a+\eta_i\geq1$.  Thus $h_i\geq c_i' g$ on all of $M$, where $c_i'=\min\{c_i,1\}>0$.  Since $g$ is complete, this lower uniform comparison implies completeness of $h_i$.

We first compare the $h_i$- and $h$-distances from $o$ to $\partial B_{r_i}$.  Any curve from $o$ to $\partial B_{r_i}$ has an initial segment, ending at its first intersection with $\partial B_{r_i}$, contained in $\overline{B_{r_i}}$, where $h_i=h$.  It follows that
\begin{equation}\label{boundary-distance-equality}
        d_{h_i}(o,\partial B_{r_i})=d_h(o,\partial B_{r_i})=D(r_i).
\end{equation}
Indeed, the first-hit argument gives $d_{h_i}(o,\partial B_{r_i})\geq D(r_i)$; the reverse inequality follows by approximating $D(r_i)$ by curves whose first boundary-hit segment is contained in $\overline{B_{r_i}}$.

Since $h_i$ is complete and $\partial B_{r_i}$ is compact, the distance from $o$ to $\partial B_{r_i}$ is attained; another application of Hopf--Rinow gives a minimizing $h_i$-geodesic.  Hence there exist a point $x_i\in\partial B_{r_i}$ and an $h_i$-minimizing geodesic $\gamma_i$ from $o$ to $x_i$ such that
$$
        L_{h_i}(\gamma_i)=d_{h_i}(o,\partial B_{r_i})=D(r_i)\leq L_0.
$$
The curve $\gamma_i$ is contained in $\overline{B_{r_i}}$.  Otherwise, if $y$ were the first point at which $\gamma_i$ hits $\partial B_{r_i}$ before the endpoint $x_i$, then the initial segment from $o$ to $y$ would have $h_i$-length strictly smaller than $d_{h_i}(o,\partial B_{r_i})$, contradicting the definition of the latter.  Hence $h_i=h$ along $\gamma_i$.

We next prove that each initial segment of $\gamma_i$ is minimizing for $h$.  Let $p$ be a point of $\gamma_i$.  Since $\gamma_i$ is globally minimizing for $h_i$, the initial segment $\gamma_i|_{o,p}$ is also globally minimizing for $h_i$; otherwise a shorter $h_i$-curve from $o$ to $p$ would shorten the whole curve from $o$ to $x_i$.  Suppose that there were an $h$-curve $\alpha$ from $o$ to $p$ with
$$
        L_h(\alpha)<L_h(\gamma_i|_{o,p}).
$$
If $\alpha$ remains in $\overline{B_{r_i}}$, then $L_{h_i}(\alpha)=L_h(\alpha)$, contradicting the $h_i$-minimality of $\gamma_i|_{o,p}$.  If $\alpha$ leaves $B_{r_i}$ before reaching $p$, let $y$ be its first intersection with $\partial B_{r_i}$.  Since $\alpha|_{o,y}$ lies in $\overline{B_{r_i}}$,
$$
        D(r_i)\leq L_h(\alpha|_{o,y})<L_h(\alpha)<L_h(\gamma_i|_{o,p})\leq D(r_i),
$$
again a contradiction.  The remaining possibility, when the first intersection with $\partial B_{r_i}$ is exactly the endpoint $p=x_i$, reduces to the first case because the curve up to its endpoint lies in $\overline{B_{r_i}}$.  Therefore every initial segment of $\gamma_i$ is $h$-minimizing.  Consequently every subsegment of $\gamma_i$ is also $h$-minimizing, because a shorter replacement of a subsegment would shorten a minimizing initial segment.

Parametrize $\gamma_i$ by $g$-arclength $s\in[0,S_i]$.  Since $x_i\in\partial B_{r_i}$, one has $S_i\geq r_i\to\infty$.  For every fixed $S>0$, the images $\gamma_i([0,S])$ lie in the compact $g$-ball $\overline{B_S}$, and for $i$ sufficiently large this compact ball is contained in $B_{r_i}$; hence on $[0,S]$ the curves are unparametrized $h$-geodesics.

Let $\varphi=(1/2)\log a$, so that $h=e^{2\varphi}g$.  An unparametrized $h$-geodesic written in $g$-unit-speed parameter satisfies the smooth pregeodesic equation
\begin{equation}\label{pregeodesic-equation}
        \nabla^g_s\dot\gamma_i
        =\nabla^g\varphi-\ip{\nabla^g\varphi}{\dot\gamma_i}_g\dot\gamma_i .
\end{equation}
On each compact ball $\overline{B_S}$ the right-hand side has uniformly bounded smooth coefficients.  Passing to a subsequence of the initial velocities and using ODE compactness, then diagonalizing in $S$, we obtain smooth convergence on compact $s$-intervals to a $g$-unit-speed curve
$$
        \gamma:[0,\infty)\to M.
$$
Equation \eqref{pregeodesic-equation} passes to the limit, so $\gamma$ is an unparametrized $h$-geodesic.

The limit is locally minimizing for $h$.  Indeed, suppose that a limit subsegment $\gamma|_{[s_1,s_2]}$ admitted an $h$-curve $\alpha$ from $\gamma(s_1)$ to $\gamma(s_2)$ such that
$$
        L_h(\alpha)\leq L_h(\gamma|_{[s_1,s_2]})-\delta
$$
for some $\delta>0$.  By the smooth convergence $\gamma_i(s_j)\to\gamma(s_j)$, $j=1,2$, we may join $\gamma_i(s_1)$ to $\gamma(s_1)$ and $\gamma(s_2)$ to $\gamma_i(s_2)$ inside small common $h$-normal balls by local $h$-geodesics with total length less than $\delta/3$.  For $i$ large, also
$$
        \left|L_h(\gamma_i|_{[s_1,s_2]})-L_h(\gamma|_{[s_1,s_2]})\right|<\delta/3 .
$$
Concatenating the two short endpoint connectors with $\alpha$ gives an $h$-curve from $\gamma_i(s_1)$ to $\gamma_i(s_2)$ shorter than $\gamma_i|_{[s_1,s_2]}$, contradicting the $h$-minimality of that subsegment.

Moreover, for every $S>0$,
$$
        L_h(\gamma|_{[0,S]})
        =\lim_{i\to\infty}L_h(\gamma_i|_{[0,S]})
        \leq \limsup_{i\to\infty} L_h(\gamma_i)
        \leq L_0 .
$$
Thus $\gamma$ has finite total $h$-length.  Define
$$
        t(s)=L_h(\gamma|_{[0,s]})=\int_0^s \sqrt{a(\gamma(\tau))}\,d\tau,
        \qquad
        T=\lim_{s\to\infty}t(s)<\infty .
$$
Since $a>0$, the function $t(s)$ is strictly increasing.  Reparametrize $\gamma$ by $h$-arclength and denote the resulting curve by $\sigma:[0,T)\to M$.  Then $\sigma$ is an $h$-unit-speed geodesic and every compact subinterval is $h$-minimizing.  Since $t(s)\to T$ as $s\to\infty$, approaching the terminal time $T$ in the $h$-arclength parameter is equivalent to letting the original $g$-arclength parameter $s$ tend to infinity.  Hence the $g$-length of $\sigma$ is infinite.

The index inequality \eqref{index-ineq-general} follows from the second variation formula on compact minimizing subsegments of $\sigma$.  More precisely, for a test function $\phi\in C_c^\infty(0,T)$, choose a compact interval containing its support.  On this interval the segment is minimizing; applying the index form to the $n-1$ variations $\phi E_j$, where $E_j$ are $h$-parallel orthonormal fields perpendicular to $\dot\sigma$, and summing in $j$ gives \eqref{index-ineq-general}.
\end{proof}

\begin{proposition}\label{prop:h-complete}
Let $5\leq n\leq 9$ and let $u>0$ solve \eqref{uRic}--\eqref{uLap} on a complete manifold $(M,g)$.  Then the conformal metric
$$
        h=ug
$$
is complete.
\end{proposition}

\begin{proof}
It suffices to argue on a connected component, since completeness is a componentwise property.  Assume by contradiction that $h$ is incomplete on such a component.  Apply Lemma \ref{lem:finite-geodesic} with $a=u$.  We obtain an $h$-unit-speed locally minimizing geodesic
$$
        \sigma:[0,T)\to M,
        \qquad T<\infty,
$$
whose $g$-length is infinite.

Let $v$ be as in \eqref{vdef}, and set
\begin{equation}\label{alpha}
        \alpha=\frac{n-4}{10-n}>0 .
\end{equation}
Since $h=ug$ and $v=u^{-(10-n)/(2(n-4))}$,
$$
        g=u^{-1}h=v^{2\alpha}h.
$$
Because $\sigma$ has unit $h$-speed,
\begin{equation}\label{infinite-glength}
        \infty=L_g(\sigma)=\int_0^T v(\sigma(t))^\alpha\,dt .
\end{equation}

On the other hand, equation \eqref{QE-v} gives, along $\sigma$,
$$
        \Ric_h(\dot\sigma,\dot\sigma)
        =m\frac{(v\circ\sigma)''}{v\circ\sigma}+\lambda .
$$
Substituting this into the index inequality \eqref{index-ineq-general} yields
\begin{equation}\label{log-from-index}
        \frac{m}{n-1}\int_0^T\frac{(v\circ\sigma)''}{v\circ\sigma}\phi^2\,dt
        \leq
        \int_0^T(\phi')^2\,dt-\frac{\lambda}{n-1}\int_0^T\phi^2\,dt
\end{equation}
for every $\phi\in C_c^\infty(0,T)$.  Thus Lemma \ref{lem:log} applies with
\begin{equation}\label{rho-def}
        \rho=\frac{m}{n-1}
        =\frac{(n-4)^2}{(10-n)(n-1)},
        \qquad
        b=-\frac{\lambda}{n-1}\geq0 .
\end{equation}
Let $p_\rho$ be given by \eqref{prho}.  We claim that
\begin{equation}\label{alphap}
        \alpha p_\rho<1 .
\end{equation}
Indeed, since $p_\rho$ is the positive root of $\rho p(p+1)=1/4$, it is enough to check that
$$
        \rho\frac1\alpha\left(\frac1\alpha+1\right)>\frac14 .
$$
Using \eqref{alpha} and \eqref{rho-def},
$$
        \rho\frac1\alpha\left(\frac1\alpha+1\right)
        =\frac{(n-4)^2}{(10-n)(n-1)}
        \cdot\frac{10-n}{n-4}\cdot\frac6{n-4}
        =\frac6{n-1}>\frac14
$$
for every $5\leq n\leq 9$.  Hence \eqref{alphap} holds.  Lemma \ref{lem:log} gives integrability on a terminal interval: for some $\varepsilon>0$,
$$
        \int_{T-\varepsilon}^T v(\sigma(t))^\alpha\,dt<\infty .
$$
On the remaining compact interval $[0,T-\varepsilon]$, the function $v\circ\sigma$ is smooth and positive, hence bounded above.  Therefore
$$
        \int_0^T v(\sigma(t))^\alpha\,dt<\infty,
$$
contradicting \eqref{infinite-glength}.  Therefore $h$ is complete.
\end{proof}

\section{A weighted maximum principle and the gradient estimate}\label{sec:maprinciple}

We next use the complete quasi-Einstein structure to prove a gradient estimate for $u$.

\begin{lemma}[Weighted cutoff maximum principle]\label{lem:weightedMP}
Let $(N^n,h)$ be complete, let $f\in C^\infty(N)$, and suppose that, for some $m>0$ and $\lambda<0$,
\begin{equation}\label{Ricfm-lower}
        \Ric_h+\Hess_h f-\frac1m df\otimes df=\lambda h .
\end{equation}
Let $G\geq0$ be smooth and satisfy
\begin{equation}\label{Gineq}
        \frac12\Delta_f G\geq G(mG+\lambda),
        \qquad
        \Delta_f=\Delta_h-\ip{\nabla_h f}{\nabla_h\cdot}_h .
\end{equation}
Then
\begin{equation}\label{GboundMP}
        G\leq -\frac{\lambda}{m}\quad\text{on }N .
\end{equation}
\end{lemma}

\begin{proof}
Equation \eqref{Ricfm-lower} says that the $m$-Bakry--Emery Ricci tensor satisfies
$$
        \Ric_f^m:=\Ric_h+\Hess_h f-\frac1m df\otimes df=\lambda h .
$$
Put $N_0=n+m$.  Since $m>0$, this is a finite synthetic dimension with $N_0>n$.  The finite-dimensional Bakry--Emery Laplacian comparison theorem for $\Ric_f^m$ with synthetic dimension $N_0$ gives, away from the cut locus of a fixed point $o$,
\begin{equation}\label{lap-comp}
        \Delta_f r\leq (N_0-1)c\coth(cr),
        \qquad
        c=\sqrt{\frac{-\lambda}{N_0-1}},
\end{equation}
where $r=d_h(o,\cdot)$.  This finite-dimensional comparison goes back to Qian \cite{Qian}; the formulation used here is, for example, the $N$-Bakry--Emery mean-curvature comparison stated in \cite[Theorem~A.1]{WeiWylie}.  In particular, $\Delta_f r\leq C_0$ on $\{r\geq1\}$ in the barrier sense.

Choose a nonincreasing smooth function $\eta:[0,\infty)\to[0,1]$ such that
$$
        \eta\equiv1\text{ on }[0,1],
        \qquad
        \eta\equiv0\text{ on }[2,\infty),
$$
with
$$
        |\eta'|^2\leq C\eta,
        \qquad
        |\eta''|\leq C .
$$
For $R\geq2$, set
$$
        \eta_R(x)=\eta(r(x)/R),
        \qquad
        H_R=\eta_R G .
$$
Since $(N,h)$ is complete, the closed ball $\overline{B_{2R}(o)}$ is compact by Hopf--Rinow.  The function $H_R$ is supported in this ball, hence it attains a maximum at some point $x_R$.  If $H_R(x_R)=0$, then $G=0$ on $B_R(o)$ and there is nothing to prove.  Otherwise $x_R$ lies where $G>0$ and $\eta_R>0$.

The standard Calabi trick handles the possible cut-locus issue.  More explicitly, if $x_R$ lies in the cut locus of $o$, one replaces $r$ near $x_R$ by the smooth upper barrier
$$
        r_\varepsilon(x)=\varepsilon+d_h(\gamma(\varepsilon),x),
$$
where $\gamma$ is a minimizing geodesic from $o$ to $x_R$.  Then $r\le r_\varepsilon$ near $x_R$ and equality holds at $x_R$.  Since $\eta$ is nonincreasing, $\eta(r_\varepsilon/R)G\le \eta(r/R)G$ near $x_R$, with equality at $x_R$; hence $x_R$ is still a local maximum for the barrier version.  Applying the following argument to $\eta(r_\varepsilon/R)G$ and letting $\varepsilon\downarrow0$, we may compute as if $r$ were smooth at $x_R$.  At $x_R$,
$$
        \nabla H_R=0,
        \qquad
        \Delta_f H_R\leq0 .
$$
Using \eqref{Gineq},
\begin{align*}
0
&\geq \Delta_f(\eta_R G) \\
&=\eta_R\Delta_f G+G\Delta_f\eta_R+2\ip{\nabla\eta_R}{\nabla G} \\
&\geq 2\eta_R G(mG+\lambda)+G\Delta_f\eta_R
      -2G\frac{\norm{\nabla\eta_R}^2}{\eta_R} .
\end{align*}
Dividing by $G>0$ and writing $H_R=\eta_R G$ at $x_R$, we obtain
\begin{equation}\label{HRineq}
        2mH_R
        \leq -2\lambda\eta_R-\Delta_f\eta_R
        +2\frac{\norm{\nabla\eta_R}^2}{\eta_R} .
\end{equation}

If $x_R\in B_R(o)$, then $\eta_R\equiv1$ in a neighbourhood of $x_R$ and the two cutoff-error terms in \eqref{HRineq} vanish; hence
$$
        H_R(x_R)\leq -\frac{\lambda}{m}.
$$
If $x_R$ lies in the transition region $R\leq r(x_R)\leq2R$, then \eqref{lap-comp} and $\eta'\leq0$ give
$$
        -\Delta_f\eta_R
        =-\frac{\eta'(r/R)}{R}\Delta_f r-\frac{\eta''(r/R)}{R^2}
        \leq \frac{C}{R}+\frac{C}{R^2},
        \qquad
        \frac{\norm{\nabla\eta_R}^2}{\eta_R}
        \leq \frac{C}{R^2} .
$$
Therefore, in all cases,
$$
        H_R(x_R)\leq -\frac{\lambda}{m}+O(R^{-1}).
$$
Since $\eta_R\equiv1$ on $B_R(o)$,
$$
        \sup_{B_R(o)}G\leq \sup_N H_R\leq -\frac{\lambda}{m}+O(R^{-1}).
$$
Letting $R\to\infty$ proves \eqref{GboundMP}.
\end{proof}

\begin{proposition}\label{prop:grad}
For $5\leq n\leq 9$, the function $u=-R>0$ satisfies
\begin{equation}\label{grad-estimate}
        \norm{\nabla_g u}_g^2
        \leq
        \frac{1}{2(n-1)}u^3 .
\end{equation}
\end{proposition}

\begin{proof}
By Proposition \ref{prop:h-complete}, $h=ug$ is complete.  Apply Lemma \ref{lem:weightedMP} to the complete weighted manifold $(M,h,e^{-f}dV_h)$ and to
$$
        G=\norm{\nabla_h w}_h^2 .
$$
The required differential inequality follows from \eqref{BochnerG}.  Hence
$$
        G\leq -\frac{\lambda}{m}
        =\frac{(10-n)^2}{8(n-1)(n-4)^2} .
$$
Since
$$
        w=-\frac{10-n}{2(n-4)}\log u
$$
and $h=ug$,
$$
        G
        =\left(\frac{10-n}{2(n-4)}\right)^2
          \frac{\norm{\nabla_g u}_g^2}{u^3} .
$$
Dividing by the prefactor gives \eqref{grad-estimate}.
\end{proof}

\section{Completion of the proof of the main theorem}\label{sec:proofmain}

We now finish the proof of Theorem \ref{thm:main}.  Recall that we are arguing by contradiction under the assumption $u=-R>0$ on the connected component under consideration.

From \eqref{grad-estimate},
\begin{equation}\label{uinv-grad}
        \norm{\nabla_g(u^{-1/2})}_g
        =\frac12 u^{-3/2}\norm{\nabla_g u}_g
        \leq \frac{1}{2\sqrt{2(n-1)}} .
\end{equation}
Fix $O\in M$.  Integrating \eqref{uinv-grad} along minimizing $g$-geodesics gives constants $c_1,c_2>0$ such that
\begin{equation}\label{u-lower}
        u(x)\geq \frac{c_1}{c_2+d_g(x,O)^2}
        \qquad\text{for all }x\in M .
\end{equation}

Let $\eta_s\in C_c^\infty(M)$ be a standard cutoff satisfying
$$
        0\leq\eta_s\leq1,
        \qquad
        \eta_s\equiv1\text{ on }B_s^g(O),
        \qquad
        \eta_s\equiv0\text{ on }M\setminus B_{2s}^g(O),
$$
and
$$
        \norm{\nabla_g\eta_s}_g\leq\frac{C}{s} .
$$
Such cutoffs may be obtained by taking the Lipschitz distance cutoff and smoothing it with only a harmless change in the gradient bound; see, for example, Greene--Wu \cite{GreeneWu}.  Alternatively, the following integration by parts can be justified directly in the weak sense with the Lipschitz cutoff.  Using \eqref{uLap}, integration by parts gives
\begin{align*}
        \frac{n-4}{4(n-1)}\int_M u^2\eta_s^2\,dV_g
        &= -\int_M (\Delta_g u)\eta_s^2\,dV_g \\
        &= 2\int_M \eta_s\ip{\nabla_g u}{\nabla_g\eta_s}_g\,dV_g \\
        &\leq 2\int_M \eta_s\norm{\nabla_g u}_g\norm{\nabla_g\eta_s}_g\,dV_g .
\end{align*}
Therefore, by \eqref{grad-estimate},
\begin{equation}\label{cutoff-ineq1}
        \frac{n-4}{4(n-1)}\int_M u^2\eta_s^2\,dV_g
        \leq
        \frac{C}{s}\int_{B_{2s}(O)\setminus B_s(O)}u^{3/2}\,dV_g .
\end{equation}
On the annulus $B_{2s}(O)\setminus B_s(O)$, the lower bound \eqref{u-lower} implies
$$
        u^{-1/2}\leq C(1+s).
$$
Hence
\begin{align*}
        \frac1s\int_{B_{2s}\setminus B_s}u^{3/2}\,dV_g
        &=\frac1s\int_{B_{2s}\setminus B_s}u^2u^{-1/2}\,dV_g \\
        &\leq C\int_{B_{2s}\setminus B_s}u^2\,dV_g .
\end{align*}
Since $u\in L^2(M,g)$, the right-hand side tends to zero as $s\to\infty$.  Therefore
\begin{equation}\label{eta-limit}
        \int_M u^2\eta_s^2\,dV_g\to0 .
\end{equation}
For any fixed $R>0$, taking $s>R$ gives
$$
        \int_{B_R(O)}u^2\,dV_g
        \leq
        \int_M u^2\eta_s^2\,dV_g .
$$
Letting $s\to\infty$ and using \eqref{eta-limit}, we obtain
$$
        \int_{B_R(O)}u^2\,dV_g=0 .
$$
Since the radius $R>0$ is arbitrary, $u\equiv0$ on the component, contradicting $u>0$.  Thus the negative scalar-curvature alternative is impossible on every connected component.  By Lemma \ref{lem:sign}, the only remaining possibility is
$$
        R\equiv0 .
$$
This proves Theorem \ref{thm:main}.

\begin{remark}
The proof uses finite energy only through the sign alternative and the final $L^2$ tail estimate.  The conformal metric $h=ug$ is tailored to the low-dimensional range: for $5\leq n\leq9$, it produces a positive finite Bakry--Emery parameter $m$ in the expanding quasi-Einstein equation.  This is opposite to the conformal metric $u^{6/(n-4)}g$ used in the high-dimensional argument of \cite{CMM2026}, which is effective for $n\geq10$ by \cite[Theorem~1.1]{CMM2026} but does not have the same sign features in dimensions $5\leq n\leq9$.
\end{remark}

\end{document}